\documentclass[a4paper,12pt]{amsart}
\calclayout
\usepackage{amssymb}
\usepackage{amsmath}
\usepackage{amsfonts}
\usepackage[all]{xy}
\usepackage{mathabx}
\usepackage[dvipdfm,CJKbookmarks,bookmarks=true,colorlinks=false]{hyperref}

\newtheorem{thm}{Theorem}[section]
\newtheorem{cor}[thm]{Corollary}
\newtheorem{lem}[thm]{Lemma}

\theoremstyle{definition}

\theoremstyle{remark}

\numberwithin{equation}{section}

\begin{document}

\title{Noncommutative $L_p$-space and operator system}

\author{Kyung Hoon Han}

\address{Department of Mathematical Sciences, Seoul National University, San 56-1 ShinRimDong, KwanAk-Gu, Seoul
151-747, Korea}

\email{kyunghoon.han@gmail.com}

\thanks{This work was supported by the BK21 project of the
Ministry of Education, Korea.}

\subjclass[2000]{46L07, 46L52, 47L07}

\keywords{noncommutative $L_p$-space, operator system}

\date{}

\dedicatory{}

\commby{}


\begin{abstract}
We show that noncommutative $L_p$-spaces satisfy the axioms of the
(nonunital) operator system with a dominating constant $2^{1 \over
p}$. Therefore, noncommutative $L_p$-spaces can be embedded into
$B(H)$ $2^{1 \over p}$-completely isomorphically and complete
order isomorphically.
\end{abstract}

\maketitle

\section{Introducton}

A unital involutive subspace of $B(H)$ had been abstractly
characterized by Choi and Effros \cite{CE}. Their axioms are based
on observations of the relationship between the unit, the matrix
order, and the matrix norm of the unital involutive subspace of
$B(H)$. Indeed, the unit and the matrix
order can be used to determine the matrix norm by applying $$\|x\|= \inf \{ \lambda>0 : - \lambda I \le \left( \begin{array}{cc} 0 & x \\
x^* & 0 \end{array} \right) \le \lambda I \}$$ for a unital
involutive subspace $X$ of $B(H)$ and $x \in M_n (X)$.

The abstract characterization of a nonunital involutive subspace
of $B(H)$ was completed by Werner \cite{W}. The axioms are based
on observations of the relationship between the matrix norm and
the matrix order of the involutive subspace of $B(H)$. Since the unit may be absent, the above equality is replaced by $$\|x\| = \sup \{|\varphi(\begin{pmatrix} 0 & x \\
x^* & 0 \end{pmatrix})| : \varphi \in M_{2n}(X)^*_{1,+} \}$$ for
an involutive subspace $X$ of $B(H)$ and $x \in M_n (X)$. We
denote by $M_{2n}(X)^*_{1,+}$ the set of positive contractive
functionals on $M_{2n}(X)$.

A complex involutive vector space $X$ is called a matrix ordered
vector space if for each $n \in \mathbb N$, there is a set
$M_n(X)_+ \subset M_n(X)_{sa}$ such that
\begin{enumerate}
\item $M_n(X)_+ \cap [-M_n(X)_+]=\{0 \}$ for all $n \in \mathbb
N$, \item $M_n (X)_+ \oplus M_m(X)_+ \subset M_{n+m}(X)_+$ for all
$m,n \in \mathbb N$, \item $\gamma^* M_m(X)_+ \gamma \subset
M_n(X)_+$ for each $m,n \in \mathbb N$ and all $\gamma \in
M_{m.,n} (\mathbb C)$.
\end{enumerate}
One might infer from these conditions that $M_n(X)_+$ is actually
a cone.

An operator space $X$ is called a matrix ordered operator space
iff $X$ is a matrix ordered vector space and for every $n \in
\mathbb N$,
\begin{enumerate}
\item the $*$-operation is an isometry on $M_n(X)$ and \item the
cones $M_n(X)_+$ are closed.
\end{enumerate}

Suppose $X$ is a matrix ordered operator space. For $x \in
M_n(X)$, the modified numerical radius is defined by $$\nu_X(x) = \sup \{|\varphi(\begin{pmatrix} 0 & x \\
x^* & 0 \end{pmatrix})| : \varphi \in M_{2n}(X)^*_{1,+} \}.$$ We
call a matrix ordered operator space an operator system iff there
is a $k>0$ such that for all $n \in \mathbb N$ and $x \in M_n(x)$,
$$\|x\| \le k \nu_X(x).$$ Since $\nu_X(x) \le \|x\|$ always holds,
we can say that an operator system is a matrix ordered operator
space such that the operator space norm and the modified numerical
radius are equivalent uniformly for all $n \in \mathbb N$.

Werner showed that $X$ is an operator system if and only if there
is a complete order isomorphism $\Phi$ from $X$ onto an involutive
subspace of $B(H)$, which is a complete topological
onto-isomorphism \cite{W}. Hence, the operator system is an
abstract characterization of the involutive subspace of $B(H)$ in
the completely isomorphic and complete order isomorphic sense.

In this paper, noncommutative $L_p$-space is meant in the sense of
Haagerup, which is based on Tomita-Takesaki theory \cite{Te1}.
While the noncommutative $L_p$-spaces arising from semifinite von
Neumann algebras are a natural generalization of classical
$L_p$-spaces \cite{S,Ta}, noncommutative $L_p$-spaces arising from
type III von Neumann algebras are quite complicated. Recently, the
reduction method approximating a general noncommutative
$L_p$-space by tracial noncommutative $L_p$-spaces has been
developed \cite{HJX}.

Noncommutative $L_p$-spaces can also be obtained by the complex
interpolation method. There exist a canonical matrix order and a
canonical operator space structure on each noncommutative
$L_p$-space. The matrix order is given by the positive cones
$L_p(M_n(\mathcal M))^+$, and the operator space structure is
given by the complex interpolation
$$M_n (L_p(\mathcal M)) = (M_n(\mathcal M), M_n(\mathcal M_*^{op}))_{1 \over
p}.$$ We refer to \cite{K,Pi1,Te2} for the details, and the
reference \cite[Section 4]{JRX} is recommended for the summary. In
particular, its discrete noncommutative vector valued $L_p$-space
$S^n_p(L_p(\mathcal M))$ is given by
$$S^n_p(L_p(\mathcal M)) = L_p (M_n (\mathcal M)).$$
We refer to \cite{Pi2} for general information on noncommutative
vector valued $L_p$-spaces.

The purpose of this paper is to show that noncommutative
$L_p$-spaces satisfy the axioms of the operator system with a
dominating constant $2^{1 \over p}$. As a corollary, we obtain the
following embedding theorem: noncommutative $L_p$-spaces can be
embedded into $B(H)$ $2^{1 \over p}$-completely isomorphically and
complete order isomorphically.

\section{Noncommutative $L_p$-space and operator system}

If $X$ is a matrix ordered operator space, we regard $S^n_p(X)$ as
an operator space having the same matrix order structure with
$M_n(X)$. That is, we do not distinguish between $S^n_p(X)$ and
$M_n(X)$ as matrix ordered vector spaces.

For a noncommutative $L_p$-space $L_p(\mathcal M)$, it is more
efficient to describe $S^n_p(L_p(\mathcal M))$ rather than
$M_n(L_p(\mathcal M))$. So, we change the modified numerical
radius into a form more adequate to noncommutative $L_p$-spaces.
For a matrix ordered operator space $X$ and $x \in M_n(X)$, we
define $\nu_X^p(x)$
as $$\nu_X^p(x) = 2^{- {1 \over p}} \sup \{|\varphi(\begin{pmatrix} 0 & x \\
x^* & 0 \end{pmatrix})| : \varphi \in S^{2n}_p (X)^*_{1,+} \},$$
where $S^{2n}_p (X)^*_{1,+}$ denotes the set of positive
contractive functionals on $S^{2n}_p (X)$. Note that a functional
$\varphi$ is positive on $S^{2n}_p (X)$ if and only if it is
positive on $M_{2n}(X)$ because $S^{2n}_p (X)$ and $M_{2n}(X)$
have the same order structure. According to the following lemma,
the uniform equivalence of a matrix norm and $\nu_X$ can be
confirmed by showing the uniform equivalence of the $S_p^n$-norm
and $\nu^p_X$.

\begin{lem}\label{numerical}
Suppose $X$ is a matrix ordered operator space, and there is a
constant $k>0$ satisfying $\| x\|_{S_p^n (X)} \le k \nu_X^p(x)$
for any $x \in M_n(X)$. Then we have $$\| x \|_{M_n(X)} \le k
\nu_X(x)$$ for any $x \in M_n(X)$. Hence, $X$ is an operator
system.
\end{lem}

\begin{proof}
For $a, b \in (S_{2p}^n)_1$, we have $$\|\begin{pmatrix} a&0 \\
0&b^* \end{pmatrix}\|_{S^{2n}_{2p}} = (\|a\|^{2p}_{S^n_{2p}} +
\|b^*\|^{2p}_{S^n_{2p}})^{1 \over 2p} \le 2^{1 \over 2p}.$$ From
\cite[Theorem 1.5]{Pi2}, the map
$$\begin{pmatrix} x_{11} & x_{12} \\ x_{21} & x_{22} \end{pmatrix} \mapsto 2^{- {1 \over p}} \varphi( \begin{pmatrix} a&0
\\ 0&b^* \end{pmatrix} \begin{pmatrix} x_{11} & x_{12} \\ x_{21} & x_{22} \end{pmatrix} \begin{pmatrix} a^*&0 \\ 0&b
\end{pmatrix})$$ for $\varphi \in
S^{2n}_p (X)^*_{1,+}$ defines a positive contractive functional on
$M_{2n}(X)$. It follows that
$$\begin{aligned} \nu_X(x) & = \sup \{ |\varphi ( \begin{pmatrix}
0&x \\ x^*&0 \end{pmatrix}
| : \varphi \in M_{2n}(X)^*_{1,+} \} \\
& \ge 2^{- {1 \over p}} \sup \{ |\varphi(\begin{pmatrix} a&0 \\
0&b^* \end{pmatrix}
\begin{pmatrix} 0&x \\ x^*&0 \end{pmatrix} \begin{pmatrix} a^*&0
\\ 0&b
\end{pmatrix})| : \varphi \in S^{2n}_p(X)^*_{1,+}, a,b \in (S^n_{2p})_1
\} \\ &= 2^{- {1 \over p}} \sup \{ |\varphi(\begin{pmatrix} 0&axb \\
b^*x^*a^*&0 \end{pmatrix}| : \varphi \in S^{2n}_p(X)^*_{1,+}, a,b
\in (S^n_{2p})_1 \} \\ &= \sup \{ \nu_X^p (axb) : a,b \in
(S^n_{2p})_1 \} \\ &\ge k^{-1} \sup \{ \|a x b\|_{S^n_p(X)} : a, b
\in (S^n_{2p})_1 \} \\ &= k^{-1} \|x\|_{M_n(X)}
\end{aligned}$$ for $x \in M_n(X)$.
For the last equality, see \cite[Lemma 1.7]{Pi2}.
\end{proof}

\begin{lem}\label{calculus}
Suppose $a$ is a bounded linear operator on a Hilbert space and $a
= v |a|$ is its polar decomposition. Let $f$ be a continuous
function defined on $[0,\infty)$ with $f(0)=0$. Then
$\begin{pmatrix} |a^*|&a \\ a^*&|a| \end{pmatrix}$ is a
positive operator and the equality $$f({1 \over 2} \begin{pmatrix} |a^*|&a \\
a^*&|a|
\end{pmatrix}) = {1 \over 2} \begin{pmatrix} f(|a^*|)&f(|a^*|)v
\\ v^*f(|a^*|)& f(|a|) \end{pmatrix}$$ holds.
\end{lem}

\begin{proof}
The positivity follows from \cite[Exercise 8.8 (vi)]{Pa}. By using
the Weierstrass theorem and the right polar decomposition $a =
|a^*| v$, it is sufficient to show that the equality
$$\begin{pmatrix} |a^*|&a \\ a^*&|a|
\end{pmatrix}^n = 2^{n-1} \begin{pmatrix} |a^*|^n&|a^*|^{n-1} a
\\ a^* |a^*|^{n-1}& |a|^n \end{pmatrix}$$ holds for $n \ge 1$.
We proceed by mathematical induction. Since
$a^*(aa^*)^ka=(a^*a)^{k+1}$ for $k \ge 0$, we have
$$a^*|a^*|^{n-1}a=(a^*a)^{{n-1 \over 2}
+1}=|a|^{n+1}.$$ Similarly, $a(a^*a)^k = (aa^*)^k a$ implies $$a
|a|^n = |a^*|^n a.$$ It follows that
$$\begin{aligned} & \begin{pmatrix} |a^*|&a \\ a^*&|a| \end{pmatrix}
\cdot 2^{n-1}
\begin{pmatrix} |a^*|^n&|a^*|^{n-1} a
\\ a^* |a^*|^{n-1}& |a|^n \end{pmatrix} \\  & = 2^{n-1} \begin{pmatrix} |a^*|^{n+1} + |a^*|^{n+1}&|a^*|^n a + a |a|^n
\\ a^* |a^*|^n + |a|a^*|a^*|^{n-1} & a^*|a^*|^{n-1}a + |a|^{n+1}
\end{pmatrix} \\ & = 2^n \begin{pmatrix} |a^*|^{n+1}&|a^*|^{n} a
\\ a^* |a^*|^n& |a|^{n+1} \end{pmatrix}.
\end{aligned}$$
\end{proof}

In order to prove that the Schatten $p$-class $S_p$ is an operator
system, Lemma \ref{calculus} is sufficient. However, in order to
prove that a general noncommutative $L_p$-space is an operator
system, an unbounded version of Lemma \ref{calculus} is needed.

We denote the Dirac measure massed at the zero point by
$\delta_0$.

\begin{lem}\label{spectral}
Let $a$ be a closed densely defined operator on a Hilbert space.
Suppose $a = v|a|$ is a polar decomposition and $|a|=
\int_0^\infty t dE$ is a spectral decomposition. Then
$$\widetilde{E} = {1 \over 2} \begin{pmatrix} vEv^* & vE \\ Ev^* & v^*v
E
\end{pmatrix} + \delta_0 \begin{pmatrix} I - {1 \over 2} vv^* & - {1 \over 2}v \\ -{1 \over 2}v^* & I - {1 \over 2}v^*v \end{pmatrix}$$
is a spectral measure corresponding to ${1 \over 2}
\begin{pmatrix} |a^*|&a \\ a^*&|a| \end{pmatrix}$.
\end{lem}

\begin{proof}
Let $$E' = {1 \over 2} \begin{pmatrix} vEv^* & vE \\ Ev^* & v^*v E
\end{pmatrix}.$$ Since $v^*v=E((0,\infty))$, $v^*v$ commutes with $E(\Delta)$ for any Borel set $\Delta$ in $[0,\infty)$. For all Borel sets $\Delta_1$ and $\Delta_2$ in $[0,
\infty)$, we have
$$\begin{aligned} & E'(\Delta_1) E'(\Delta_2) \\ = & {1 \over 4}
\begin{pmatrix} v E(\Delta_1) v^*&vE(\Delta_1) \\
E(\Delta_1)v^*&v^*vE(\Delta_1) \end{pmatrix} \begin{pmatrix} v E(\Delta_2) v^*&vE(\Delta_2) \\
E(\Delta_2)v^*&v^*vE(\Delta_2) \end{pmatrix} \\ = & {1 \over 4}
\begin{pmatrix} vE(\Delta_1)v^*vE(\Delta_2)v^* +
vE(\Delta_1)E(\Delta_2)v^* &
vE(\Delta_1)v^*vE(\Delta_2)+vE(\Delta_1)v^*vE(\Delta_2) \\
E(\Delta_1)v^*vE(\Delta_2)v^* + v^*vE(\Delta_1)E(\Delta_2)v^* &
E(\Delta_1)v^*vE(\Delta_2)+v^*vE(\Delta_1)v^*vE(\Delta_2)
\end{pmatrix} \\ =& {1 \over 2} \begin{pmatrix} v E(\Delta_1 \cap \Delta_2) v^*& vE(\Delta_1 \cap \Delta_2) \\
E(\Delta_1 \cap \Delta_2)v^*&v^*vE(\Delta_1 \cap \Delta_2)
\end{pmatrix} \\ =&E'(\Delta_1 \cap \Delta_2).
\end{aligned}$$
Taking $\Delta_1=\Delta_2$, we see that $E'$ is projection valued.
The countable additivity with respect to strong operator topology
is obvious. Since $E(\{0\})=I-v^*v$, we have $$E'(\{0\})=0.$$ From
these, it is easy to check that $\widetilde E$ satisfies all the
conditions of a spectral measure.

We still need to show that $${1 \over 2} \langle \begin{pmatrix} |a^*|&a \\
a^*&|a| \end{pmatrix} \begin{pmatrix} h\\k
\end{pmatrix} |
\begin{pmatrix} h\\k \end{pmatrix} \rangle = \int_0^\infty z d
\widetilde{E}_{\xi,\xi}$$ for $\xi=\begin{pmatrix} h\\k
\end{pmatrix}$ and $h \in D(a^*), k \in D(a).$ For the left-hand side, we compute
$$\begin{aligned} \langle \begin{pmatrix}
|a^*|&a \\ a^*&|a| \end{pmatrix} \begin{pmatrix} h\\k
\end{pmatrix} |
\begin{pmatrix} h\\k \end{pmatrix} \rangle & = \langle |a^*| h | h
\rangle + \langle a k | h \rangle + \langle a^* h | k \rangle +
\langle |a| k | k \rangle \\ &= \langle v|a|v^*h | h \rangle +
\langle v|a| k | h \rangle + \langle |a|v^*h | k \rangle + \langle
v^*v|a| k | k \rangle \\ & = \int_0^\infty t d E_{v^*h,v^*h} +
\int_0^\infty t d E_{k,v^*h} \\ & \qquad \ + \int_0^\infty t d
E_{v^*h,k} + \int_0^\infty t d E_{k,v^*vk}.
\end{aligned}$$ For the right-hand side, we compute $$\begin{aligned}
E'_{\xi,\xi} (\Delta) & = {1 \over 2} \langle \begin{pmatrix} v E(\Delta) v^*&vE(\Delta) \\
E(\Delta)v^*&v^*vE(\Delta) \end{pmatrix} \begin{pmatrix} h\\k
\end{pmatrix} |
\begin{pmatrix} h\\k \end{pmatrix} \rangle \\ & = {1 \over
2}(\langle vE(\Delta)v^*h|h \rangle +\langle vE(\Delta)k|h \rangle
+\langle E(\Delta)v^*h|k \rangle + \langle v^*v E(\Delta)k|k
\rangle)
\\ &= {1 \over 2}(E_{v^*h,v^*h}(\Delta) + E_{k,v^*h}(\Delta) +
E_{v^*h,k}(\Delta) + E_{k,v^*vk}(\Delta)).
\end{aligned}$$
\end{proof}

\begin{lem}\label{unbounded}
Suppose $\mathcal M$ is a semifinite von Neumann algebra with a
faithful semifinite normal trace $\tau$ and $a$ is a
$\tau$-measurable operator with its polar decomposition $a = v
|a|$. Let $f$ be a continuous function defined on $[0,\infty)$
with $f(0)=0$. Then $\begin{pmatrix} |a^*|&a \\ a^*&|a|
\end{pmatrix}$ is a positive self-adjoint $\tau_2$-measurable operator and
the equality
$$f({1 \over 2} \begin{pmatrix} |a^*|&a \\
a^*&|a|
\end{pmatrix}) = {1 \over 2} \begin{pmatrix} f(|a^*|)&f(|a^*|)v
\\ v^*f(|a^*|)& f(|a|) \end{pmatrix}$$ holds.
\end{lem}

\begin{proof}
Because we have $$\begin{aligned} \tau_2(\tilde E ((n,\infty)) & =
\tau ( {1 \over 2} v E((n ,\infty))v^*) + \tau ({1 \over 2} v^*v
E(n,\infty)) \\ & = \tau(E(n,\infty)) \to 0,\end{aligned}$$ $\begin{pmatrix} |a^*|&a \\ a^*&|a|
\end{pmatrix}$ is $\tau_2$-measurable.

Suppose $E$ is a spectral measure corresponding to $|a|$. We let
$$a_n = a E([0,n]).$$ Then we have $$|a_n|=(E([0,n]a^*a
E([0,n])^{1 \over 2}=|a| E([0,n]).$$ Since $a_n v^* = v |a|
E([0,n]) v^* \ge 0$ and $(a_n v^*)^2 = a E([0,n]) v^*v E([0,n])
a^* = a E([0,n]) a^*$, we have $$a_n v^* = (a E([0,n]) a^*)^{1
\over 2} = |a^*_n|.$$ We compute $$\begin{aligned} & {1 \over 2} \begin{pmatrix} |a^*|&a \\
a^*&|a| \end{pmatrix} \widetilde{E}([0,n]) \\ = & {1 \over 4} \begin{pmatrix} |a^*|&a \\
a^*&|a| \end{pmatrix} \begin{pmatrix} vE([0,n])v^* & vE([0,n]) \\
E([0,n])v^* & v^*v E([0,n]) \end{pmatrix} \\ = & {1 \over 4}
\begin{pmatrix} |a^*|vE([0,n])v^* + a E([0,n])v^* & |a^*|vE([0,n]) + av^*vE([0,n]) \\ a^*vE([0,n])v^*+|a|E([0,n])v^* & a^*vE([0,n])+|a|v^*vE([0,n])\end{pmatrix}  \\
= & {1 \over 2} \begin{pmatrix} |a_n^*| & a_n \\ a_n^* & |a_n|
\end{pmatrix}.
\end{aligned}$$
Since $|a_n|$ is bounded, we can apply Lemma \ref{calculus}. The
partial isometry in the polar decomposition of $a_n$ is $v_n :=
vE([0,n])$. It
follows that $$\begin{aligned} f({1 \over 2} \begin{pmatrix} |a^*|&a \\
a^*&|a| \end{pmatrix}) \widetilde{E}([0,n]) & = f( {1 \over 2} \begin{pmatrix} |a_n^*|&a_n \\
a_n^*&|a_n| \end{pmatrix}) \\& = {1 \over 2}
\begin{pmatrix} f(|a_n^*|)&f(|a_n^*|)v_n
\\ v_n^*f(|a_n^*|)& f(|a_n|) \end{pmatrix}.
\end{aligned}$$

Since $v^*|a_n^*|^k v = v^* (a_n v^*)^k v=|a_n|^k$ for any $k \ge
1$, we have $$v^*f(|a_n^*|)v=f(|a_n|).$$ Similarly, we also have
$$v f(|a_n|)v^*=f(|a^*_n|).$$ The spectral measure for
$|a^*|$ is $vEv^*+\delta_0 (I-vv^*)$. It follows that
$$\begin{aligned} & {1 \over 2} \begin{pmatrix} f(|a^*|)&f(|a^*|)v
\\ v^*f(|a^*|)& f(|a|) \end{pmatrix} \widetilde{E}([0,n]) \\ = &
{1 \over 4}
\begin{pmatrix} 2 f(|a^*|)vE([0,n])v^* & 2 f(|a^*|) v E([0,n]) \\ v^* f(|a^*|)vE([0,n])v^* + f(|a|)E([0,n])v^* & v^*f(|a^*|)vE([0,n])+f(|a|)v^*vE([0,n])\end{pmatrix}
\\ = & {1 \over 4} \begin{pmatrix} 2 f(|a^*|vE([0,n])v^*) & 2 f(|a^*| v E([0,n]) v^*) v_n \\ v_n^* f(|a^*|vE([0,n])v^*) + v_n^*vf(|a|E([0,n]))v^* & v^*f(|a^*|vE([0,n])v^*)v+f(|a|E([0,n]))\end{pmatrix}
\\ = & {1 \over 2} \begin{pmatrix} f(|a_n^*|)&f(|a_n^*|)v_n
\\ v_n^*f(|a_n^*|)& f(|a_n|) \end{pmatrix}.
\end{aligned}$$
Because the set $\bigcup_{n=1}^\infty {\rm ran}
\widetilde{E}([0,n])$ is $\tau_2$-dense, we obtain the desired
result from \cite[Proposition I.12]{Te1}.
\end{proof}

\begin{lem}\label{norm}
Suppose $\mathcal M$ is a von Neumann algebra with a faithful
semifinite normal weight $\varphi$. If $a$ is an element of
Haagerup's noncommutative $L_p$-space $L_p(\mathcal M)$, then we
have $$\begin{pmatrix} |a^*|&a \\ a^*&|a|
\end{pmatrix} \in L_p(M_2(\mathcal M))_+ \qquad {\rm and} \qquad \| \begin{pmatrix} |a^*|&a \\ a^*&|a|
\end{pmatrix} \|_{L_p(M_2(\mathcal M))} = 2 \|a \|_{L_p(\mathcal
M)}.$$
\end{lem}

\begin{proof}
Let $\theta$ be the dual action of $\mathbb R$ on $\mathcal M
\rtimes_{\sigma^\varphi} \mathbb R$. Then $id_{M_2} \otimes
\theta$ is the dual action of $\mathbb R$ on $M_2(\mathcal M)
\rtimes_{\sigma^{tr \otimes \varphi}} \mathbb R$.
Because $$\begin{aligned} (id_{M_2} \otimes \theta_s) \begin{pmatrix} |a^*|&a \\
a^*&|a| \end{pmatrix} & = \begin{pmatrix} \theta_s( |a^*|)&
\theta_s (a) \\ \theta_s (a^*)& \theta_s
(|a|)\end{pmatrix} \\ & = e^{-{s \over p}} \begin{pmatrix} |a^*|&a \\
a^*&|a| \end{pmatrix},
\end{aligned}$$
we have $$\begin{pmatrix} |a^*|&a \\ a^*&|a|
\end{pmatrix} \in L_p(M_2(\mathcal M))_+.$$

From Lemma \ref{unbounded}, it follows that $$\begin{aligned} \|
\begin{pmatrix} |a^*|&a \\ a^*&|a|
\end{pmatrix} \|_{L_p(M_2(\mathcal M))}^p & = \|
{\begin{pmatrix} |a^*|&a \\ a^*&|a|
\end{pmatrix}}^p \|_{L_1(M_2(\mathcal M))} \\ & =
2^{p-1} \|\begin{pmatrix} |a^*|^p&|a^*|^pv \\ v^*|a^*|^p& |a|^p
\end{pmatrix}\|_{L_1(M_2(\mathcal M))} \\ & = 2^{p-1} (\||a^*|^p\|_{L_1(\mathcal M)} +
\||a|^p\|_{L_1(\mathcal M)}) \\ & = 2^p \|a \|_{L_p(\mathcal
M)}^p.
\end{aligned}$$
\end{proof}

\begin{thm} \label{main}
Suppose $\mathcal M$ is a von Neumann algebra with a faithful
semifinite normal weight $\varphi$. Its noncommutative $L_p$-space
$L_p(\mathcal M)$ is an operator system with
$$\|a\|_{M_n(L_p (\mathcal M))} \le 2^{1 \over p} \nu_{L_p(\mathcal M)}
(a),\qquad a \in M_n(L_p (\mathcal M)).$$
\end{thm}

\begin{proof}
Let $a \in L_p(M_n(\mathcal M))$ with its right polar
decomposition $a = |a^*| v$. Because
$$\begin{aligned} \|a\|_{L_p(M_n(\mathcal M))} & = \| |a^*|v\|_{L_p(M_n(\mathcal M))} \\ & \le \|a^*\|_{L_p(M_n(\mathcal M))},
\end{aligned}$$
the involution is an isometry on $L^p(M_n( \mathcal M))$, which
can be identified with $S^n_p(L^p(\mathcal M))$. From \cite[Lemma
1.7]{Pi2}, the involution is also an isometry on $M_n(L^p(\mathcal
M))$. The H$\ddot{\rm o}$lder inequality and \cite[Proposition
II.33]{Te1} show that the positive cone $L^p(M_n(\mathcal M))_+$
is closed. Hence, the noncommutative $L_p$-space is a matrix
ordered operator space.

Let $a \in S^n_p(L_p(\mathcal M))$ and ${1 \over p} + {1 \over q}
= 1$. From Lemma \ref{norm}, it follows that
$$\begin{aligned} \|a\|_{L_p(M_n{(\mathcal M}))} & = \sup \{
|{\rm tr_{\mathcal M}}(b^*a)| : b \in L_q(M_n{(\mathcal M}))_1\} \\
& = \sup \{ {\rm Re}~ {\rm tr_{\mathcal M}} (b^*a) : b \in L_q(M_n{(\mathcal M}))_1\} \\
& = \sup \{ {1 \over 2} ( {\rm tr_{\mathcal M}}(b^*a) + {\rm tr_{\mathcal M}}(ba^*)) : b \in L_q(M_n(\mathcal M))_1\} \\
& = \sup \{ {1 \over 2} {\rm tr}_{M_2(\mathcal M)} \begin{pmatrix} ba^* & |b^*|a \\
|b|
a^* & b^* a \end{pmatrix} : b \in L_q(M_n(\mathcal M))_1\} \\
& = \sup \{ {1 \over 2} {\rm tr}_{M_2(\mathcal M)}(\begin{pmatrix} |b^*|&b \\
b^*&|b| \end{pmatrix}
\begin{pmatrix} 0&a \\ a^*&0 \end{pmatrix}) : b \in L_q(M_n(\mathcal M))_1\} \\
& \le \sup \{ \varphi (\begin{pmatrix} 0&a \\ a^*&0 \end{pmatrix})
: \varphi \in L_p(M_{2n}{(\mathcal M}))^*_{1,+} \} \\
& = 2^{1 \over p} \nu^p_{L_p(\mathcal M)} (a).
\end{aligned}$$
Thus, from Lemma \ref{numerical}, we obtain the desired result.
\end{proof}

By combining this with \cite[Theorem 4.15]{W}, we obtain the
following embedding theorem.

\begin{cor}
Noncommutative $L_p$-spaces can be embedded into $B(H)$ $2^{1
\over p}$-completely isomorphically and complete order
isomorphically.
\end{cor}

For a $C^*$-algebra $A$, we define an involution on its dual space
$A^*$ as $$\varphi^*(a) = \varphi (a^*) ^*, \qquad \varphi \in
A^*, a \in A.$$ Along the line of operator space duality, we
identify $M_n(A^*)$ and $M_n(A)^*$ algebraically by using the
parallel duality pairing
$$\langle [\varphi_{ij}],[a_{ij}] \rangle = \sum_{i,j} \varphi_{ij}(a_{ij}), \qquad \varphi_{ij} \in
A^*, a_{ij} \in A.$$ We can now define the positive cone
$M_n(A^*)_+$. However, the duality between the von Neumann algebra
$M_n(\mathcal M)$ and its noncommutative $L_1$-space
$L^1(M_n(\mathcal M))$ is given by the trace duality pairing
$$\langle [a_{ij}],[b_{ij}] \rangle = {\rm tr}_{M_n} \otimes {\rm
tr}_{\mathcal M}([a_{ij}][b_{ij}]) = \sum_{i,j} {\rm tr}_{\mathcal
M} (a_{ij} b_{ji}).$$ Hence, we must be more careful.

For an operator space $V$, its opposite operator space $V^{op}$ is
defined by $$\|[v_{ij}^{op}] \|_{M_n(V^{op})} = \|
[v_{ji}]\|_{M_n(V)}.$$ Note that $V^{op}$ and $V$ are the same
normed spaces. For a matrix ordered operator space $X$, we define
its opposite matrix ordered operator space $X^{op}$ as the
opposite operator space with the matrix order $$[x_{ij}^{op}] \ge
0 \Leftrightarrow [x_{ji}] \ge 0.$$

A functional $\varphi : M_n(X) \to \mathbb C$ can be written as
$$\varphi = [\varphi_{ij}], \qquad \varphi([x_{ij}]) = \sum_{i,j}
\varphi_{ij} (x_{ij}).$$ Let $\tilde \varphi = [\varphi_{ji}]$.
Then we have $$\varphi([x_{ij}]) = \tilde \varphi ([x_{ji}]).$$ A
functional $\varphi : M_n(X) \to \mathbb C$ is positive and
contractive if and only if $\tilde \varphi : M_n(X^{op}) \to
\mathbb C$ is positive and contractive. From this, we see that $X$
is an operator system if and only if $X^{op}$ is an operator
system.

\begin{cor}
The dual space of a $C^*$-algebra is an operator system.
\end{cor}

\begin{proof}
The bidual $A^{**}$ of a $C^*$-algebra $A$ can be identified with
the enveloping von Neumann algebra of $A$. The opposite matrix
ordered operator space $(A^*)^{op}$ is completely isometric and
completely order isomorphic to $L_1(A^{**})$ by \cite[Theorem
II.7]{Te1}. Hence, the dual space $A^*$ is an operator system.
\end{proof}

Finally, we show that the constant $2^{1 \over p}$ in the
inequality $$\|a\|_{L_p(M_n{(\mathcal M}))} \le 2^{1 \over p}
\nu^p_{L_p(\mathcal M)} (a)$$ is the best one. However, this does
not imply that $2^{1 \over p}$ is the best constant in Theorem
\ref{main}. Unexpectedly, it is achieved by the 2-dimensional
commutative case. We consider an element $(1,1)$ in $\ell^2_p$. It
is sufficient to show that
$$\nu^p_{\ell^2_p} ((1,1)) \le 1.$$ For ${1 \over p} + {1 \over
q} = 1$, we let
$$\begin{pmatrix} a_1
& \cdot & b_1 & \cdot \\ \cdot & a_2 & \cdot & b_2 \\
\bar b_1 & \cdot & d_1 & \cdot \\ \cdot & \bar b_2 & \cdot & d_2
\end{pmatrix} \in (S^4_q)_{1,+}.$$ We have $$\begin{aligned} 0 & \le {\rm tr} (\begin{pmatrix} a_1
& \cdot & b_1 & \cdot \\ \cdot & a_2 & \cdot & b_2 \\
\bar b_1 & \cdot & d_1 & \cdot \\ \cdot & \bar b_2 & \cdot & d_2
\end{pmatrix}
\begin{pmatrix} 1 & \cdot & -1 & \cdot \\ \cdot & 1 & \cdot & -1 \\
-1 & \cdot & 1 & \cdot \\ \cdot & -1 & \cdot & 1 \end{pmatrix})
\\ & = a_1+a_2+d_1+d_2-b_1-b_2-\bar b_1 -\bar b_2.
\end{aligned}$$ It follows that $$\begin{aligned} {\rm tr} (\begin{pmatrix} a_1
& \cdot & b_1 & \cdot \\ \cdot & a_2 & \cdot & b_2 \\
\bar b_1 & \cdot & d_1 & \cdot \\ \cdot & \bar b_2 & \cdot & d_2
\end{pmatrix}
\begin{pmatrix} \cdot & \cdot & 1 & \cdot \\ \cdot & \cdot & \cdot & 1 \\
1 & \cdot & \cdot & \cdot \\ \cdot & 1 & \cdot & \cdot
\end{pmatrix})
& = b_1+b_2+\bar b_1 +\bar b_2 \\ & \le {1 \over
2}(a_1+a_2+d_1+d_2+b_1+b_2+\bar b_1+\bar b_2) \\ & = {1
\over 2}{\rm tr}(\begin{pmatrix} a_1 & \cdot & b_1 & \cdot \\
\cdot & a_2 & \cdot & b_2 \\ \bar b_1 & \cdot & d_1 & \cdot \\
\cdot & \bar b_2 & \cdot & d_2
\end{pmatrix}
\begin{pmatrix} 1 & \cdot & 1 & \cdot \\ \cdot & 1 & \cdot & 1 \\
1 & \cdot & 1 & \cdot \\ \cdot & 1 & \cdot & 1
\end{pmatrix}) \\
& \le {1 \over 2} \|\begin{pmatrix} 1 & \cdot & 1 & \cdot \\ \cdot & 1 & \cdot & 1 \\
1 & \cdot & 1 & \cdot \\ \cdot & 1 & \cdot & 1
\end{pmatrix}\|_{S^4_p} \\ & ={1 \over 2} \|\begin{pmatrix} 1 & 1 & \cdot & \cdot \\ 1 & 1 & \cdot & \cdot \\
\cdot & \cdot & 1 & 1 \\ \cdot & \cdot & 1 & 1
\end{pmatrix}\|_{S^4_p} \\ & = 2^{1 \over p}. \end{aligned}$$
In the same manner, we can also show that $${\rm tr}
(\begin{pmatrix} a_1 & \cdot & b_1 & \cdot \\ \cdot & a_2 & \cdot
& b_2 \\ \bar b_1 & \cdot & d_1 & \cdot \\ \cdot & \bar b_2 &
\cdot & d_2
\end{pmatrix}
\begin{pmatrix} \cdot & \cdot & 1 & \cdot \\ \cdot & \cdot & \cdot & 1 \\
1 & \cdot & \cdot & \cdot \\ \cdot & 1 & \cdot & \cdot
\end{pmatrix}) \ge - 2^{1 \over p}.$$ We obtain the inequality $$\nu^p_{\ell^2_p} ((1,1)) \le 1.$$


\end{document}